\documentclass[12pt,a4paper,reqno]{amsart}
\usepackage[utf8]{inputenc}
\usepackage[english]{babel}
\usepackage{amsmath}
\usepackage{amsfonts}
\usepackage{amssymb}
\usepackage{color}

\usepackage{amsthm}

\newtheorem{lemma}{Lemma} [section]

\newtheorem{theorem}{Theorem} [section]
\newtheorem{corollary}{Corollary}[section]

\usepackage{graphics}
\usepackage{epsfig}

\usepackage{url}
\usepackage{hyperref}

\usepackage[a4paper,top=2.5cm,bottom=2.8cm,
left=3.2cm,right=3.2cm]{geometry}
\markleft{\hfill \textsc{A Triad of Circles Associated with a Triangle} \hfill}
\begin{document}

Sangaku Journal of Mathematics (SJM) \copyright SJM \\
ISSN 2534-9562 \\
Volume 7 (2023), pp. 54--70 \\
Received 18 July 2023 Published on-line 3 August 2023 \\ 
web: \url{http://www.sangaku-journal.eu/} \\
\copyright The Author(s) This article is published 
with open access.\footnote{This article is distributed under the terms of the Creative Commons Attribution License which permits any use, distribution, and reproduction in any medium, provided the original author(s) and the source are credited.} \\
\bigskip
\bigskip

\begin{center}
	{\Large \textbf{A Triad of Circles Associated with a Triangle}} \\
	\bigskip

	\textsc{Ercole Suppa$^a$ and Stanley Rabinowitz$^b$} \\

	$^a$ Via B. Croce 54, 64100 Teramo, Italia \\
	e-mail: \href{mailto:ercolesuppa@gmail.com}{ercolesuppa@gmail.com} \\
	web: \url{http://www.esuppa.it/} \\

	$^b$ 545 Elm St Unit 1,  Milford, New Hampshire 03055, USA \\
	e-mail: \href{mailto:stan.rabinowitz@comcast.net}{stan.rabinowitz@comcast.net}\footnote{Corresponding author} \\
	web: \url{http://www.StanleyRabinowitz.com/} \\
		
\end{center}
\bigskip

\textbf{Abstract.} We study some properties of a triad of circles associated with a triangle.
Each circle is inside the triangle, tangent to two sides of the triangle, and
externally tangent to the circle on the third side as diameter.
In particular, we find a nice relation involving the radii of the inner and outer Apollonius circles of the three circles in the triad.



\long\edef\void#1{}



\section{Introduction}

\bigskip
\textbf{Notation.} Throughout this paper, we will use the following notation, where $\triangle ABC$ is a fixed
acute triangle in the plane. We let $a=BC$, $b=CA$, $c=AB$, $r$ is the inradius of $\triangle ABC$, $R$ is the circumradius of $\triangle ABC$, $p=\frac{a+b+c}{2}$, $\Delta=[ABC]$ is the area of the triangle,
and $S=2\Delta$. We also let $I$ denote the incenter of $\triangle ABC$.

The semicircle erected inwardly on side $BC$ will be named $\omega_a$ as shown in Figure~\ref{fig:circle} (left).
Semicircles $\omega_b$ and $\omega_c$ are defined similarly.
The circle inside $\triangle ABC$, tangent to sides $AB$ and $AC$, and externally tangent to
semicircle $\omega_a$ will be named $\gamma_a$.
Circles $\gamma_b$ and $\gamma_c$ are defined similarly.
The radii of circles $\gamma_a$, $\gamma_b$, and $\gamma_c$ are denoted by $\rho_a$, $\rho_b$, and $\rho_c$, respectively.
The centers of these circles are named $D$, $E$, and $F$, respectively, as shown in Figure~\ref{fig:circle} (right).

\begin{figure}[ht]
\centering
\includegraphics[scale=1]{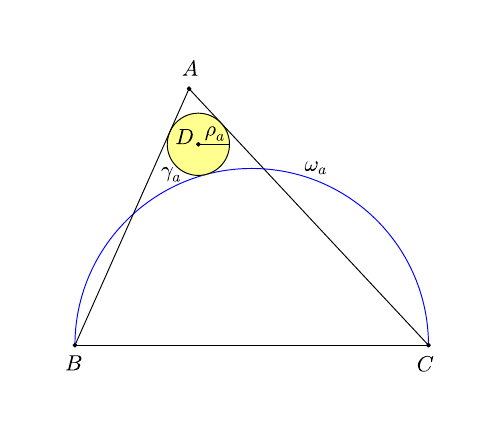}
\qquad
\includegraphics[scale=1]{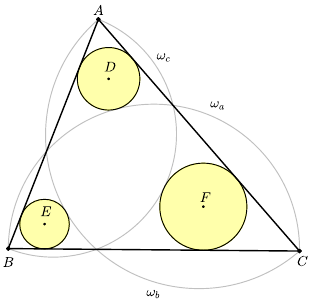}
\caption{}
\label{fig:circle}
\end{figure}

For purposes of this paper, these three circles will be called the \emph{triad of circles} associated with $\triangle ABC$.

This triad of circles appears in a Sangaku described in \cite{Honma} and reprinted in \cite[problem~6]{Okumura}.
The statement in the Sangaku is given as Theorem~\ref{thm:Sangaku}.

\begin{theorem}\label{thm:Sangaku}
For the triad of circles associated with $\triangle ABC$, we have
$$r=\frac12\left(\rho_a+\rho_b+\rho_c+\sqrt{\rho_a^2+\rho_b^2+\rho_c^2}\right).$$
\end{theorem} 

A proof of this result can be found in \cite{Sup}.
A variant on this result when the triangle is not acute can also be found in \cite{Sup}.

It is the purpose of this paper to give other properties of such a triad of circles.

\section{Known Results}

Before giving new results, we summarize some of the properties already known about the triad of circles.
The following five theorems come from \cite{Sup}.

\begin{theorem}\label{thm:rho}
For the triad of circles associated with $\triangle ABC$, we have
$$\rho_a=r\left(1-\tan\frac{A}{2}\right).$$
Similar formulas hold for $\rho_b$ and $\rho_c$.
\end{theorem} 

\begin{theorem}\label{thm:inradius}
Let $P$ and $Q$ be the feet of the perpendiculars from $D$ and $I$ to side $AC$, respectively.
Then $PQ=IQ$. (Figure~\ref{fig:inradius})
\end{theorem}

\begin{figure}[ht]
\centering
\includegraphics[scale=0.5]{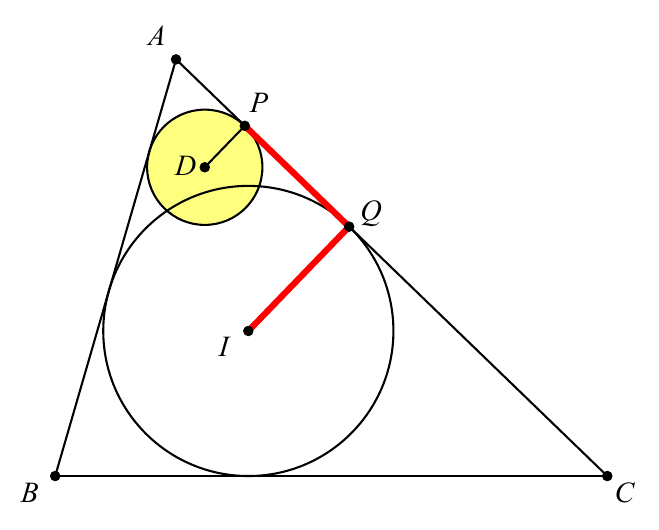}
\caption{red lengths are equal}
\label{fig:inradius}
\end{figure}

\begin{theorem}\label{thm:tangents}
The lengths of the common external tangents between any two circles
of the triad are equal. The common length is $2r$. (Figure~\ref{fig:tangents})
\end{theorem}

\begin{figure}[ht]
\centering
\includegraphics[scale=0.5]{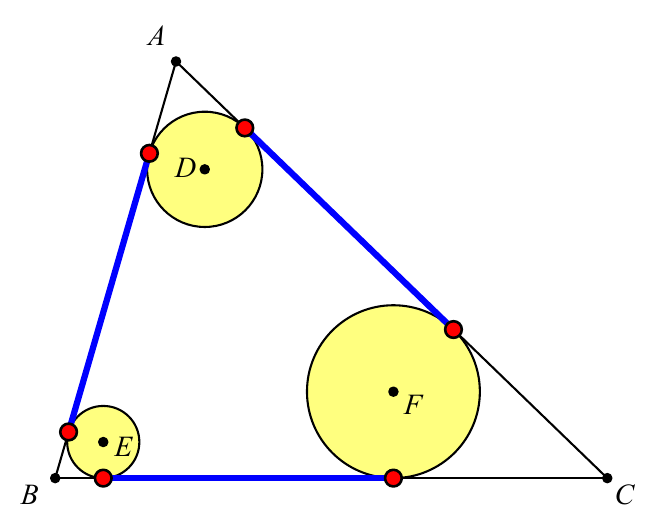}
\caption{blue lengths are equal}
\label{fig:tangents}
\end{figure}

\begin{theorem}\label{thm:concyclic}
The six points of contact of the triad of circles associated with $\triangle ABC$
lie on a circle with center $I$ and radius $r\sqrt2$. (Figure~\ref{fig:concyclic})
\end{theorem}

\begin{figure}[ht]
\centering
\includegraphics[scale=0.5]{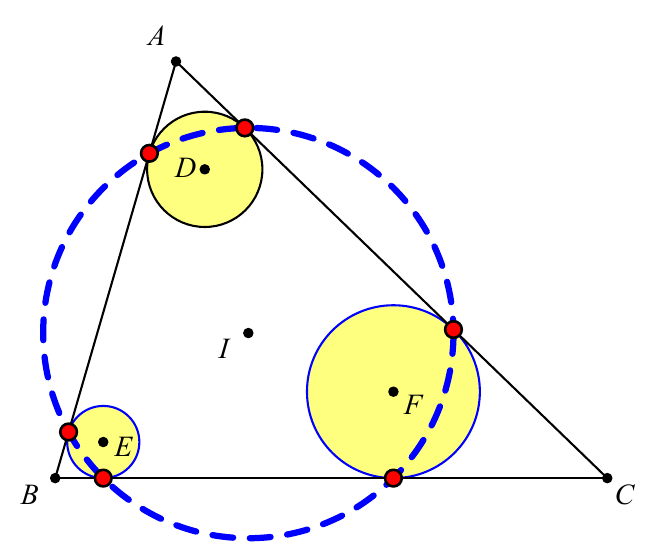}
\caption{}
\label{fig:concyclic}
\end{figure}

This circle will be called the \emph{contact circle}.

The following corollary follows immediately from Theorem~\ref{thm:concyclic}.

\begin{corollary}\label{thm:annulus}
In Figure~\ref{fig:annulus} showing the contact circle and the incircle,
the green area is equal to the blue area.
\end{corollary}

\begin{figure}[ht]
\centering
\includegraphics[scale=0.5]{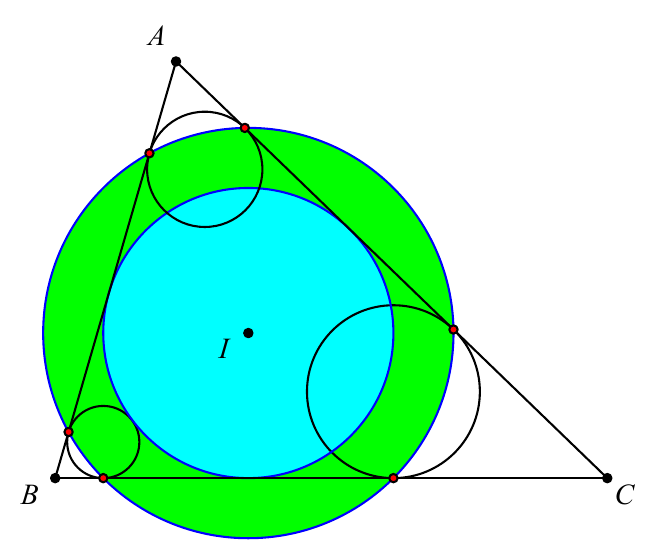}
\caption{green area = blue area}
\label{fig:annulus}
\end{figure}

\begin{theorem}\label{thm:Suppa}
For the triad of circles associated with $\triangle ABC$, we have
\begin{equation}
\rho_a^2+\rho_b^2+\rho_c^2=\frac{r^2(p-4R-r)^2}{p^2}.\label{eq:1}
\end{equation}
\end{theorem}

The following result comes from \cite{ETC1123} where it is stated that the result is due to Tomasz Cie\'sla.

\begin{theorem}\label{thm:Ciesla}
Let $T_a$, $T_b$, and $T_c$ be the touch points of the circles in the triad
with their corresponding semicircles as shown in Figure~\ref{fig:Ciesla}.
Then $AT_a$, $BT_b$, and $CT_c$ are concurrent.
\end{theorem}

\begin{figure}[ht]
\centering
\includegraphics[scale=1.3]{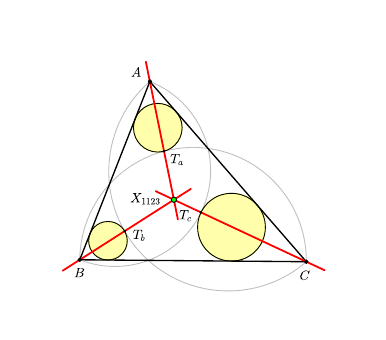}
\caption{}
\label{fig:Ciesla}
\end{figure}

The point of concurrence is catalogued as point $X_{1123}$ in the Encyclopedia of Triangle Centers \cite{ETC1123}.
Since reference \cite{ETC1123} does not include a proof of this result, we will give our own proof
later in Section~\ref{section:Paasche} of this paper.

The point $X_{1123}$ is known as the \emph{Paasche point} of the triangle because Paasche
proved the following result in \cite{Paasche}.

\begin{theorem}\label{thm:Paasche}
Congruent circles with centers $A_1$ and $A_2$ touch each other externally at point $A'$ outside $\triangle ABC$.
Circle $(A_1)$ is tangent to $AB$ and $BC$.
Circle $(A_2)$ is tangent to $AC$ and $BC$.
Points $B'$ and $C'$ are defined similarly (Figure~\ref{fig:Paasche}).
Then $AA'$, $BB'$, and $CC'$ are concurrent.
\end{theorem}

\begin{figure}[ht]
\centering
\includegraphics[scale=0.4]{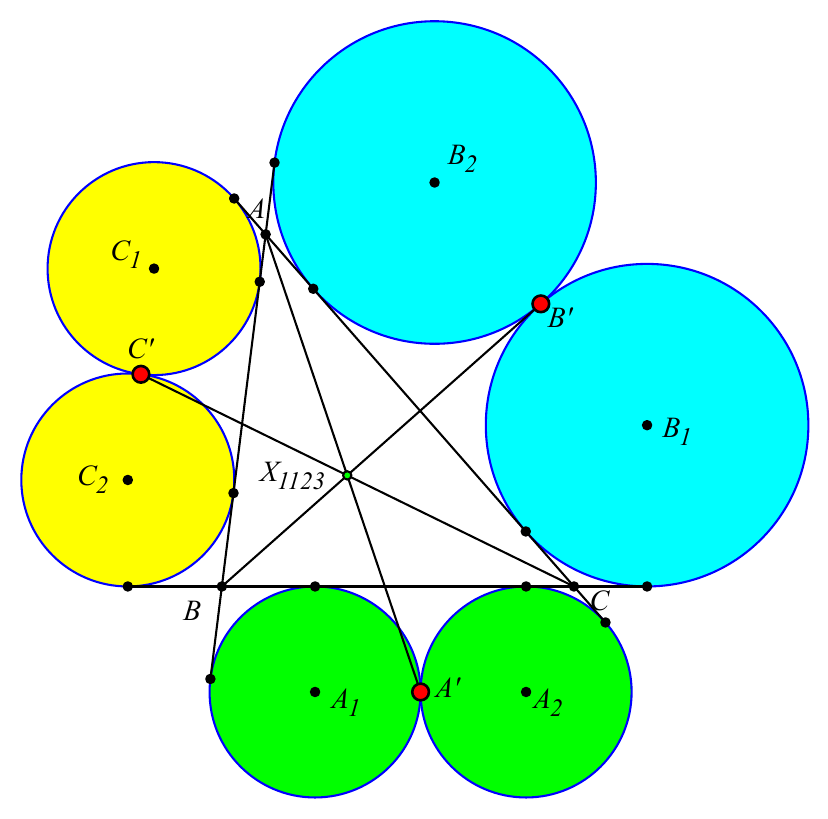}
\caption{}
\label{fig:Paasche}
\end{figure}

\textbf{Remark.}
The same result is true if the pairs of congruent circles
are inside the triangle instead of outside.
Figure~\ref{fig:Paasche2} illustrates this.
(Only the two congruent circles tangent to side $BC$ are shown.)
This result comes from \cite[Art.~3.5.4, ex.~4c]{Yiu}.

\begin{figure}[ht]
\centering
\includegraphics[scale=0.7]{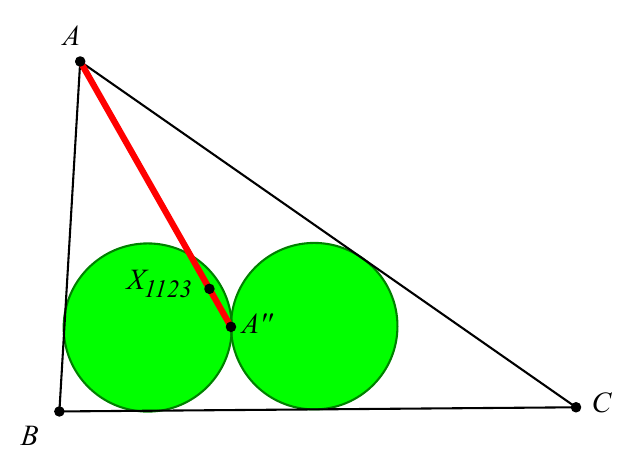}
\caption{}
\label{fig:Paasche2}
\end{figure}

The Paasche point can also be characterized as follows according to \cite{Capitan}.

\begin{theorem}\label{thm:Paasche3}
In $\triangle ABC$, let $\omega_a$, $\omega_b$, and $\omega_c$ be the circles constructed using
sides $BC$, $CA$, and $AB$, respectively, as diameters.
Let $\Omega$ be the circle internally tangent to $\omega_a$, $\omega_b$, and $\omega_c$.
Let $A'$ be the touch point between $\omega_a$ and $\Omega$.
Points $B'$ and $C'$ are defined similarly (Figure~\ref{fig:Paasche3}).
Then $AA'$, $BB'$, and $CC'$ are concurrent at $X_{1123}$, the Paasche point of $\triangle ABC$.
\end{theorem}

\begin{figure}[ht]
\centering
\includegraphics[scale=0.7]{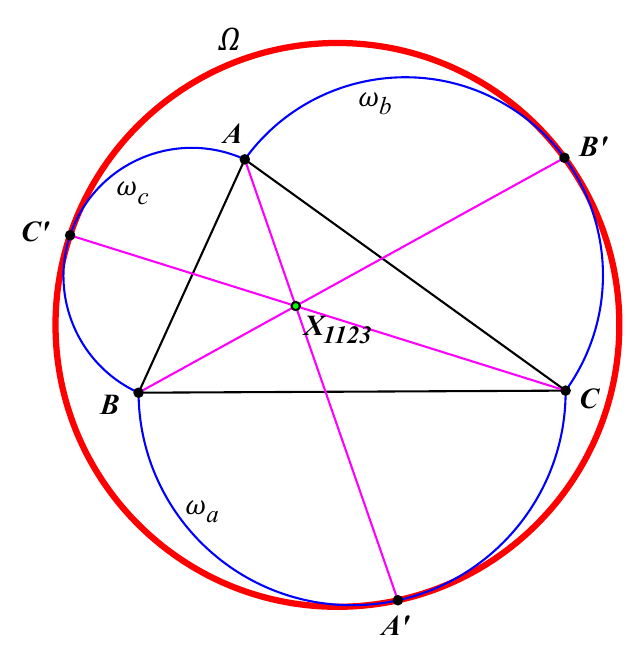}
\caption{}
\label{fig:Paasche3}
\end{figure}

A circle that is tangent to three given circles is called an \emph{Apollonius circle} of those three circles.

If all three circles lie inside an Apollonius circle, then the Apollonius circle
is called the \emph{outer Apollonius circle} of the three circles.
The outer Apollonius circle surrounds the three circles and is internally tangent to all three.

If all three circles lie outside an Apollonius circle, then the Apollonius circle
is called the \emph{inner Apollonius circle} of the three circles.
The inner Apollonius circle will either be internally tangent to the three given circles
or it will be externally tangent to all the circles.
The following theorem comes from \cite{ETC52805}.

\begin{theorem}\label{thm:tangentCircles}
For the triad of circles associated with $\triangle ABC$,
the inner Apollonius circle of $\gamma_a$, $\gamma_b$, $\gamma_c$, is internally tangent
to the inner Apollonius circle of $\omega_a$, $\omega_b$, $\omega_c$ (Figure~\ref{fig:tangentCircles}).
\end{theorem}

\textbf{Remark.} The inner Apollonius circle of $\gamma_a$, $\gamma_b$, $\gamma_c$
is known as the 1st Miyamoto-Moses-Apollonius circle and the outer Apollonius circle of $\gamma_a$, $\gamma_b$, $\gamma_c$
is known as the 2nd Miyamoto-Moses-Apollonius circle (see \cite{ETC52805}).

\begin{figure}[ht]
\centering
\includegraphics[scale=0.5]{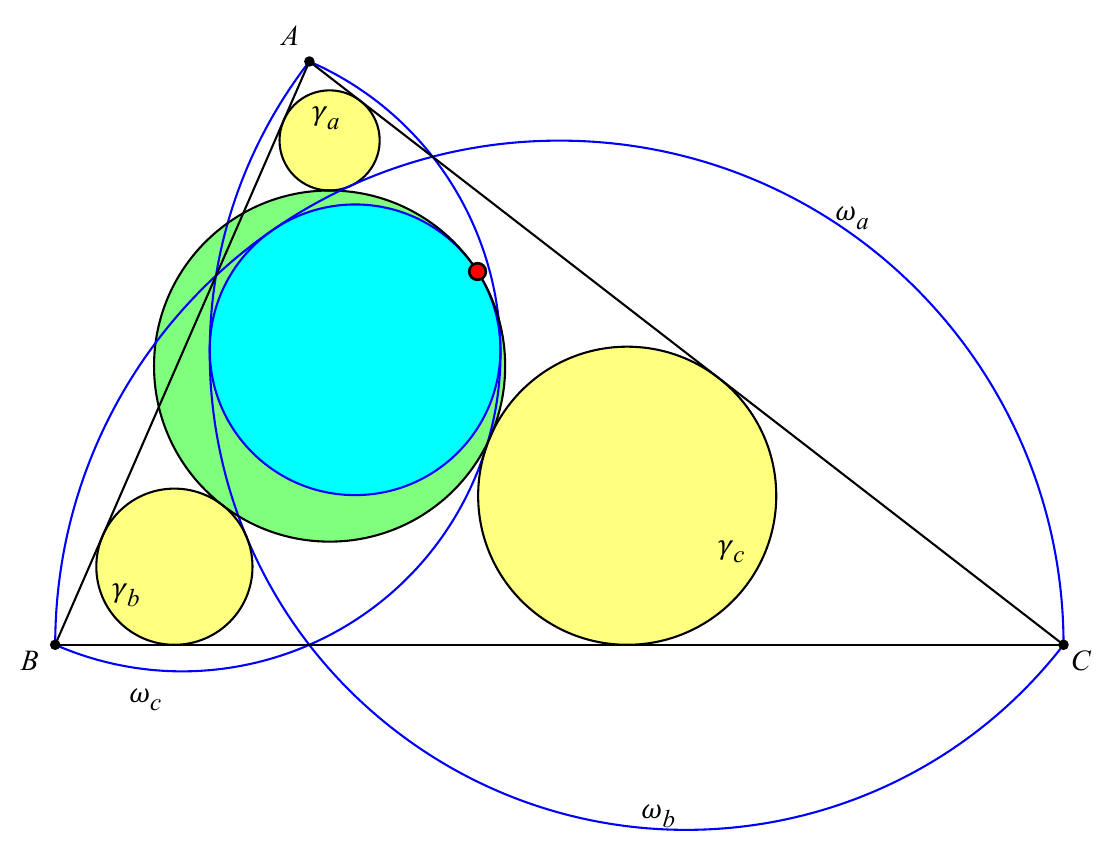}
\caption{}
\label{fig:tangentCircles}
\end{figure}

\section{Metric Relationships involving $\rho_a$, $\rho_b$, $\rho_c$}

In addition to Theorem~\ref{thm:Sangaku} and Theorem~\ref{thm:Suppa}, the following
symmetric relationship involving $\rho_a$, $\rho_b$, and $\rho_c$ holds.

\begin{theorem}
For the triad of circles associated with $\triangle ABC$, we have
\begin{equation*}
\rho_a\rho_b+\rho_b\rho_c+\rho_c\rho_a-2r(\rho_a+\rho_b+\rho_c)+2r^2=0.
\end{equation*}
\end{theorem}

\begin{proof}
From Theorem~\ref{thm:rho}, we have $\rho_a=r(1-\tan\frac A2)$,
$\rho_b=r(1-\tan\frac B2)$, and $\rho_c=r(1-\tan\frac C2)$.
Substituting these values into the expression
$$\rho_a\rho_b+\rho_b\rho_c+\rho_c\rho_a-2r(\rho_a+\rho_b+\rho_c)+2r^2$$
and simplifying (using \textsc{Mathematica}), shows that the expression is equal to
$$-r^2\cos\left(\frac{A+B+C}{2}\right)\sec\frac A2\sec\frac B2\sec \frac C2.$$
Since $A+B+C=\pi$, this expression is equal to 0.
\end{proof}

\begin{theorem}
For the triad of circles associated with $\triangle ABC$, we have
\begin{equation*}
p(r-\rho_a)(r-\rho_b)(r-\rho_c)=r^4.
\end{equation*}
\end{theorem}

\begin{proof}
This result follows from Theorem~\ref{thm:rho} and the trigonometric identity
$$\tan\frac A2\tan\frac B2\tan\frac C2=\frac rp$$
which comes from \cite[p.~358]{Andreescu}.
\end{proof}

\begin{theorem}\label{thm:rho1} 
For the triad of circles associated with $\triangle ABC$, we have
\begin{equation}
\rho_a=\frac{\Delta-(p-b)(p-c)}{p}.\label{eqn:1}
\end{equation}
Similar formulas hold for $\rho_b$ and $\rho_c$.
\end{theorem}

\begin{proof} 
Using Theorem~\ref{thm:rho} and the well-known identities
\begin{equation*}
  \tan \frac{A}{2}=\frac{r}{p-a} \qquad\hbox{and}\qquad r=\frac{\Delta}{p}, 
\end{equation*}
we get
\begin{align} \label{eqn:2}
   \rho_a&=r \left(1-\frac{r}{p-a}\right)=r-\frac{r^2}{p-a}=\frac{\Delta}{p}-\frac{\Delta^2}{p^2 (p-a)} \\ 
				 &=\frac{\Delta}{p}-\frac{p(p-a)(p-b)(p-c)}{p^2(p-a)}=\frac{\Delta-(p-b)(p-c)}{p}.
\end{align}
This complete the proof.
\end{proof}

\section{Barycentric coordinates of centers of $\gamma_a$, $\gamma_b$, $\gamma_c$}

\begin{theorem}\label{thm2} 
The barycentric coordinates of the center of $\gamma_a$ are
\begin{equation*}
      D=aS+2(p-b)(p-c)(b+c):bS-2b(p-b)(p-c):cS-2c(p-b)(p-c).
\end{equation*}
\end{theorem}
\begin{proof}

Let $y$ be the distance between $D$ and the sideline $BC$. Summing the areas of triangles $DBC$, $DCA$ and $DAB$ we obtain
\begin{equation}\label{eqn:3}
   ay+b\rho_a +c\rho_a=2\Delta.
\end{equation}
Plugging (\ref{eqn:1}) into (\ref{eqn:3}) we get
\begin{gather} 
  ay+(b+c)\cdot \frac{\Delta-(p-b)(p-c)}{p}=S \nonumber \\
	ay+(b+c)\cdot \frac{S-2(p-b)(p-c)}{2p}=S \nonumber \\
	2pay+(b+c)S-2(p-b)(p-c)(b+c)=(a+b+c)S \nonumber \\
	2pay=2(p-b)(p-c)(b+c)+aS \nonumber \\
	ay=\frac{aS+2(p-b)(p-c)(b+c)}{2p} \label{eqn:4}
\end{gather}
By using (\ref{eqn:1}) and (\ref{eqn:4}), we obtain
\begin{align*}
   D&=\Delta DBC:\Delta DCA:DAB=ay:b\rho_a:c\rho_a\\
	  &=aS+2(p-b)(p-c)(b+c):bS-2b(p-b)(p-c):cS-2c(p-b)(p-c)
\end{align*}
which are the desired barycentric coordinates.	
\end{proof}

\goodbreak

\begin{theorem}\label{thm2a} 
The radical center of circles $\gamma_a$, $\gamma_b$ and $\gamma_c$ is the Gergonne point of $\triangle ABC$.
\end{theorem}

\begin{proof}
Using \textsc{Mathematica} and the package \texttt{baricentricas.m}\footnote{The package \texttt{baricentricas.m} written by F.J.G.Capitan can be freely downloaded from \url{http://garciacapitan.epizy.com/baricentricas/}}, it can be proved that 
\begin{itemize}
	\item[(a)] the radical axis of $\gamma_a$ and $\gamma_b$ is $(p-a)x-(p-b)y=0$;
	\item[(b)] the radical axis of $\gamma_b$ and $\gamma_c$ is $(p-b)y-(p-c)z=0$;
	\item[(c)] the radical axis of $\gamma_c$ and $\gamma_a$ is $(p-a)x-(p-c)z=0$.
\end{itemize}
An easy verification shows that these radical axes concur at 
\begin{equation*}
  G_e=(p-b)(p-c):(p-c)(p-a):(p-a)(p-b),  
\end{equation*}
which is the Gergonne point of $\triangle ABC$.
\end{proof}

\medskip
We can also give a purely geometric proof.
\medskip

\begin{proof}
Let the incircle touch the sides of $\triangle ABC$ at $Q_a$, $Q_b$, and $Q_c$ as shown in
Figure~\ref{fig:radicalCenter}.
From Theorem~\ref{thm:inradius}, $Q_aE_a=Q_aF_a=r$.
Thus, the tangents from $Q_a$ to $\gamma_b$ and $\gamma_c$ are equal.
Since $AD_c=AD_b$ and $D_cE_c=D_bF_b$ (Theorem~\ref{thm:tangents}), this means $AE_c=AF_b$.
Hence the tangents from $A$ to $\gamma_b$ and $\gamma_c$ are equal.
The radical axis of circles $\gamma_b$ and $\gamma_c$ is the locus of points
such that the lengths of the tangents to the two circles from that point are equal.
The radical axis of two circles is a straight line.
Therefore, the radical axis of circles $\gamma_b$ and $\gamma_c$ is $AQ_a$,
the Gergonne cevian from $A$.

\smallskip
Similarly, the radical axis of circles $\gamma_a$ and $\gamma_c$ is the Gergonne
cevian from $B$ and the radical axis of circles $\gamma_a$ and $\gamma_b$ is the Gergonne
cevian from $C$. Hence, the radical center of the triad of circles is the intersection point
of the three Gergonne cevians, namely, the Gergonne point of $\triangle ABC$.
\end{proof}

\begin{figure}[ht]
\centering
\includegraphics[scale=0.8]{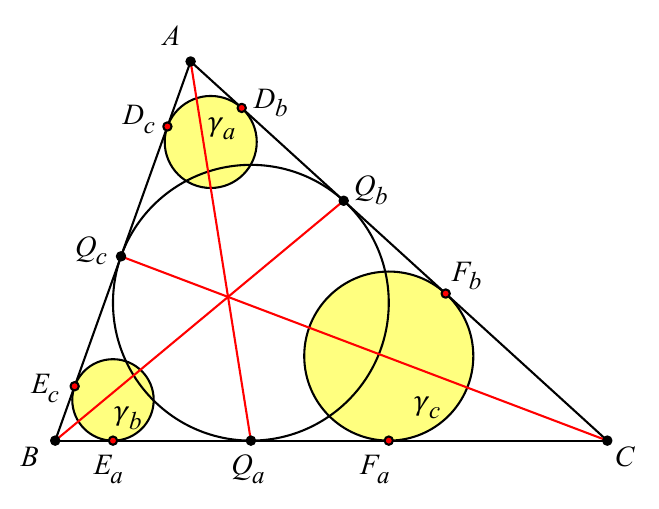}
\caption{}
\label{fig:radicalCenter}
\end{figure}

\goodbreak

\section{A concurrence at the Paasche point.}

\label{section:Paasche}

In \cite{ETC1123} the following result is stated.

\begin{theorem}\label{thm3} 
Suppose that $ABC$ is an acute triangle. Let $\gamma_a$ be the circle touching $CA$ and $AB$ from inside of $ABC$ and also externally tangent to the semicircle of diameter $BC$,
in point $T_a$. Define $T_b$ and $T_c$ cyclically (Figure~\ref{fig:3}). Then $ABC$ is perspective to $T_aT_bT_c$, and the perspector is $X_{1123}$, the Paasche point of $\triangle ABC$.

\begin{figure}[ht]
\centering
\includegraphics[scale=1.5]{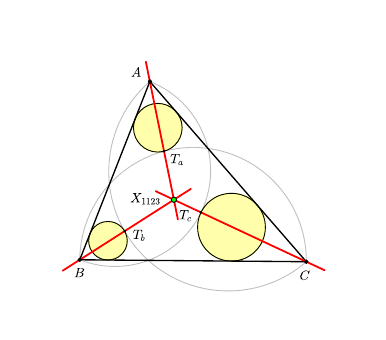}
\caption{}
\label{fig:3}
\end{figure}
\end{theorem}

\begin{proof}
We use homogeneous barycentric coordinates with respect to the triangle $ABC$. Let $M_a$ be the midpoint of $BC$. 
The point $T_a$ divides the segment joining the centers of the circles $\gamma_a$ and $\omega_a$ in the ratio $\rho_a:\frac{a}{2}$. 
Using theorem \ref{thm2} we have that the sum of coordinates of $D$ is
\begin{align*}
     &aS+2(p-b)(p-c)(b+c)+bS-2b(p-b)(p-c)+cS-2c(p-b)(p-c)\\
    =&(a+b+c)S+2(p-b)(p-c)(b+c)-2(b+c)(p-b)(p-c)\\
		=&(a+b+c)S=2pS.
\end{align*}
Therefore, by writing the coordinates of $M_a$ in the form $M_a=0:pS:pS$, we get
\begin{equation*}
   T_a=\frac{a}{2}\cdot D+\rho_a\cdot M_a.
\end{equation*}
It follows that the first coordinate of $T_a=x_a:y_a:z_a$ is given by
\begin{align*}
   x_a&=\frac{a}{2}\left(aS+2(p-b)(p-c)(b+c)\right)+\rho_a\cdot 0\\
	    &=\frac{a}{2}\left(aS+2(p-b)(p-c)(b+c)\right).
\end{align*}
In a similar way we find that 
\begin{align*}
   y_a&=\frac{a}{2}\left(bS-2b(p-b)(p-c)\right)+\rho_a pS\\
	    &=\frac{1}{2}\left(S-2(p-b)(p-c)\right)(ab+S)
\end{align*}
and
\begin{align*}
   z_a&=\frac{a}{2}\left(cS-2c(p-b)(p-c)\right)+\rho_a pS\\
	    &=\frac{1}{2}\left(S-2(p-b)(p-c)\right)(ac+S).    
\end{align*}
Hence
\begin{align*}
   T_a=&a^2S+2a(p-b)(p-c)(b+c):\\
	     &\left(S-2(p-b)(p-c)\right)(ab+S):\\
			 &\left(S-2(p-b)(p-c)\right)(ac+S).    
\end{align*}
 The equation of line $AT_a$ is $z_ay+y_az=0$, i.e.
\begin{equation*}
   AT_a: \quad (ac+S)y-(ab+S)z=0.
\end{equation*}
The cyclic substitution $a\to b$, $b\to c$, $c\to a$ gives
\begin{align*}
   &BT_b: \quad (bc+S)x-(ab+S)z=0,\\
	 &CT_c: \quad (bc+S)x-(ac+S)y=0.
\end{align*}
A direct verification shows that $AT_a$, $BT_b$, and $CT_c$ concur at the Paasche point 
\begin{equation*}
 X_{1123}=(a b + S) (a c + S):(a b + S) (b c + S):(a c + S) (b c + S). \qedhere
\end{equation*}
\end{proof}

\section{Apollonius circles of $\gamma_a$, $\gamma_b$, $\gamma_c$}

In order to find the radii of the inner and outer Apollonius circles tangent to $\gamma_a$, $\gamma_b$,
and $\gamma_c$
we will use the method explained in \cite{Ste} and some preliminary lemmas. The more complicated calculations are 
performed with \textsc{Mathematica}.

\begin{lemma}\label{lemma1}
If $u=EF$, $v=DF$, $w=DE$ are the distances between the centers of the circles $\gamma_a$, $\gamma_b$, and $\gamma_c$,
we have
\begin{align*}
	u^2&=\frac{a (b+c-a) \left(a^2+a b+a c-2 b^2+4 b c-2 c^2\right)}{(a+b+c)^2},\\
  v^2&=\frac{b (a-b+c) \left(-2 a^2+a b+4 a c+b^2+b c-2 c^2\right)}{(a+b+c)^2},\\
	w^2&=\frac{c (a+b-c) \left(-2 a^2+4 a b+a c-2 b^2+b c+c^2\right)}{(a+b+c)^2}.	
\end{align*}
\end{lemma}

\begin{proof}
From Theorem \ref{thm:rho1}, we have
\begin{equation*}
   \rho_a=\frac{\Delta-(p-b) (p-c)}{p},
\end{equation*}
and similarly
\begin{equation*}
   \rho_b=\frac{\Delta-(p-a)(p-c)}{p}, \qquad \rho_c=\frac{\Delta-(p-a)(p-b)}{p}.
\end{equation*}

\begin{figure}[ht]
\centering
\includegraphics[scale=0.3]{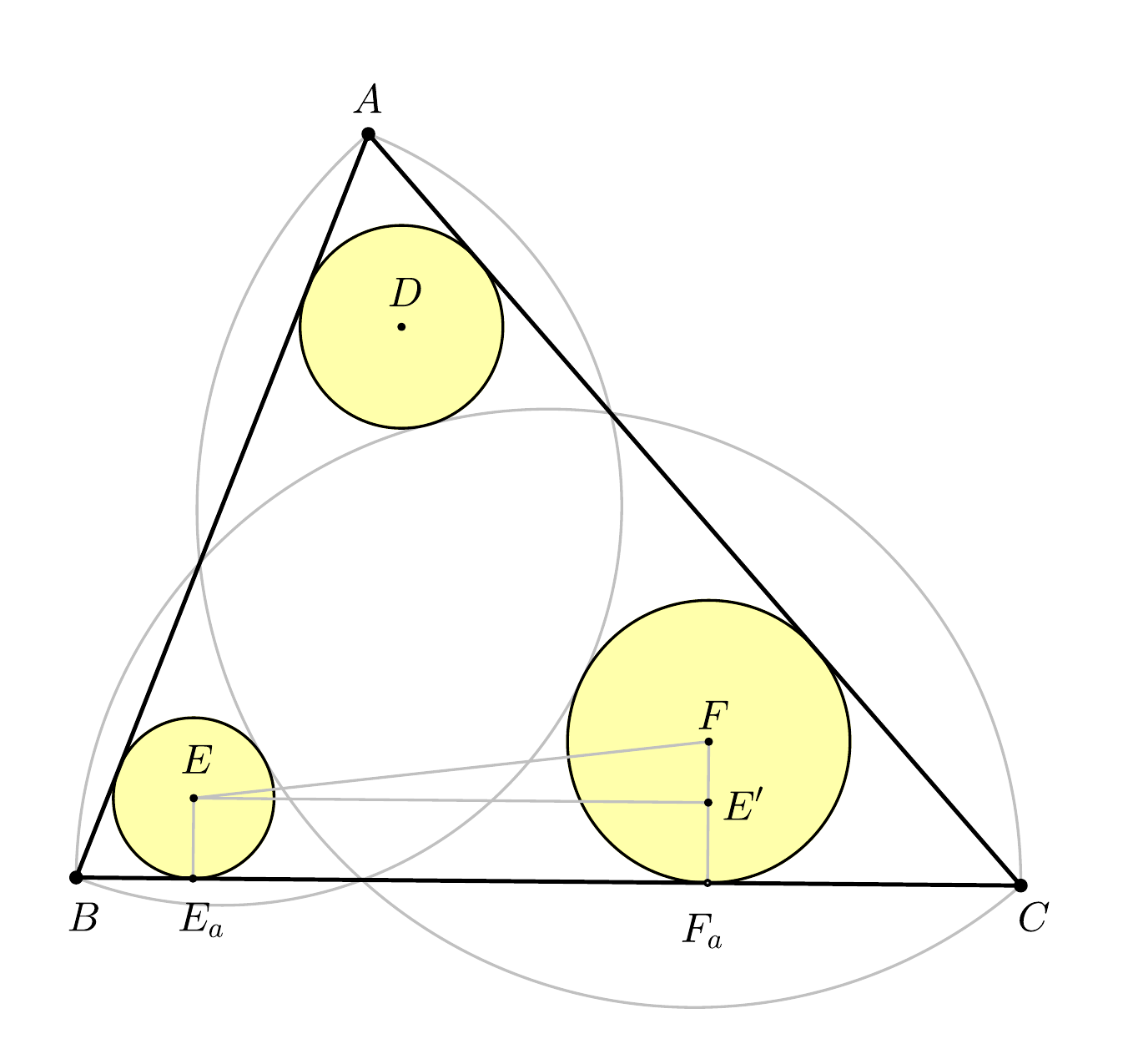}
\caption{}
\label{fig:4}
\end{figure}

Assume, without loss of generality, that $\rho_b<\rho_c$. Let $E'$ be the foot of the perpendicular from $E$ to $FF_a$. 
Applying the Pythagorean Theorem to triangle $EFE'$ (see figure \ref{fig:4}), taking into account that $EE'=E_aF_a=2r$ and $FE'=\rho_c-\rho_b$, we obtain
\begin{align*}
	 u^2&=EF^2=E_aF_a^2+(FF_a-EE_a)^2=(2r)^2+(\rho_c-\rho_b)^2\\
	    &=4r^2+\left(\frac{(p-a)(p-c)-(p-a)(p-b)}{p}\right)^2\\
			&=4\cdot\frac{\Delta^2}{p^2}+\frac{(p-a)^2(b-c)^2}{p^2}\\
			&=\frac{4p(p-a)(p-b)(p-c)+(p-a)^2(b-c)^2}{p^2}\\
			&=\frac{a(b+c-a) \left(a^2+a b+a c-2 b^2+4 b c-2 c^2\right)}{(a+b+c)^2}.
\end{align*}
The formulas relating to $v^2$ and $w^2$ can be proved in a similar way.
\end{proof}

\begin{lemma}\label{lemma2}
If $\theta$, $\varphi$, $\psi$ are three positive real numbers such that $\theta+\varphi+\psi=360^\circ$, then we have
\begin{equation*}
   \cos^2 \theta+\cos^2 \varphi+\cos^2 \psi-2\cos \theta\cos \varphi\cos \psi=1.
\end{equation*} 
\end{lemma}


\begin{proof} Using the addition formulas we have
\begin{align*}
   &\cos^2\theta+\cos^2\varphi+\cos^2\psi-2\cos\theta\cos\varphi\cos\psi\\
	=&\cos^2\theta+\cos^2\varphi+\cos^2(360^\circ-\theta-\varphi)-2\cos\theta\cos\varphi\cos(360^\circ-\theta-\varphi)\\
	=&\cos^2\theta+\cos^2\varphi+\cos^2(\theta+\varphi)-2\cos\theta\cos\varphi\cos(\theta+\varphi)\\
	=&\cos^2\theta+\cos^2\varphi+\left(\cos\theta\cos\varphi-\sin\theta\sin\varphi\right)^2-2\cos\theta\cos\varphi\left(\cos\theta\cos\varphi-\sin\theta\sin\varphi\right)\\
	=&\cos^2\theta+\cos^2\varphi+\cos^2\theta\cos^2\varphi+\sin^2\theta\sin^2\varphi-2\cos^2\theta\cos^2\varphi\\
	=&\cos^2\theta+\cos^2\varphi+\sin^2\theta\sin^2\varphi-\cos^2\theta\cos^2\varphi\\
	=&\cos^2\theta(1-\cos^2\varphi)+\cos^2\varphi+\sin^2\theta\sin^2\varphi\\
	=&\sin^2\varphi(\cos^2\theta+\sin^2\theta)+\cos^2\varphi\\
	=&\sin^2\varphi+\cos^2\varphi=1.\qedhere
\end{align*}
\end{proof}

\begin{theorem}\label{thm4} 
Let $\rho_i$ be the radius of the inner Apollonius circle externally tangent to $\gamma_a$, $\gamma_b$,
and $\gamma_c$ (see figure \ref{fig:14}).
Then
\begin{equation*}
   \rho_i^2=\rho_a^2+\rho_b^2+\rho_c^2.
\end{equation*}

\begin{figure}[ht]
\centering
\includegraphics[scale=1]{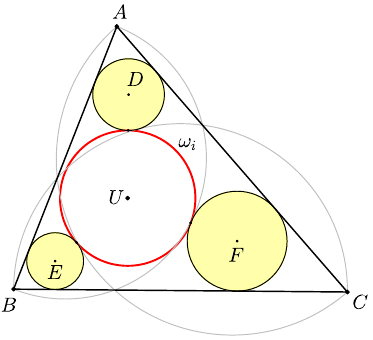}
\caption{}
\label{fig:14}
\end{figure}

\end{theorem}

\begin{proof}
Let $U$ be the center of the inner Apollonius circle and let $x=\rho_i$. 
Let us consider the angles $\varphi_a=\angle EUF$, $\varphi_b=\angle FUD$,
and $\varphi_c=\angle DUE$.
 
Since $\varphi_a+\varphi_b+\varphi_c=360^\circ$ by Lemma \ref{lemma2} we have
\begin{equation}\label{eq:13}
   \cos^2 \varphi_a+\cos^2 \varphi_b+\cos^2 \varphi_c-2\cos \varphi_a\cos \varphi_b\cos \varphi_c=1.
\end{equation}
If we substitute
\begin{equation}\label{eq:14}
   t_a=\sin^2 \frac{\varphi_a}{2}, \qquad t_b=\sin^2 \frac{\varphi_b}{2}, \qquad t_c=\sin^2 \frac{\varphi_c}{2}
\end{equation}  
in (\ref{eq:13}), we obtain
\begin{equation}\label{eq:15}
   t_a^2+t_b^2+t_c^2-2\left(t_at_b+t_bt_c+t_ct_a\right)+4t_at_2t_c=0.
\end{equation}
Since $UD=x+\rho_a$, $UE=x+\rho_b$, $UF=x+\rho_c$, the Law of Cosines yields
\begin{gather}
   \cos \varphi_a=\frac{\left(x+\rho_b\right)^2+\left(x+\rho_c\right)^2-u^2}{2(x+\rho_b)(x+\rho_c)} \quad \Rightarrow\nonumber\\
	 \quad t_a=\frac{1-\cos \varphi_1}{2}=\frac{u^2-\left(\rho_b-\rho_c\right)^2}{4(x+\rho_b)(x+\rho_c)},\label{eq:16}
\end{gather} 
and analogously
\begin{equation}\label{eq:17}
   t_b=\frac{v^2-\left(\rho_c-\rho_a\right)^2}{4(x+\rho_c)(x+\rho_a)}, \qquad t_c=\frac{w^2-\left(\rho_a-\rho_b\right)^2}{4(x+\rho_a)(x+\rho_b)}.
\end{equation}
Plugging (\ref{eq:16}) and (\ref{eq:17}) in (\ref{eq:15}) and using Lemma \ref{lemma1}, after a straightforward calculation, we get an equation of the form $f(x)g(x)=0$, 
where
\begin{equation}\label{eq:18}
   f(x)=6(a+b+c)x+2ab+2bc+2ac-a^2-b^2-c^2+12\Delta
\end{equation}
and 
\begin{equation}\label{eq:19}
   g(x)=2(a+b+c)x+a^2+b^2+c^2-2ab-2ac-2bc+4\Delta.  
\end{equation}
The root of (\ref{eq:18}) is
\begin{equation*}
x=\frac{a^2+b^2+c^2-2ab-2bc-2ca-12\Delta}{6(a+b+c)}=-\frac{2r(4R+r)+6\Delta}{6p}<0.   
\end{equation*}
The root of (\ref{eq:19}) is
\begin{equation*}
x=\frac{-a^2-b^2-c^2+2ba+2bc+2ca-4\Delta}{2 (a + b + c)}=\frac{r (4 R+r-p)}{p}>0.
\end{equation*} 
Therefore, discarding the negative root, we have $\rho_i=\frac{r(4R+r-p)}{p}$. 
Hence, taking into account equation~(\ref{eq:1}), we get 
\begin{equation*}
 \rho_i^2=\rho_a^2+\rho_b^2+\rho_c^2.\qedhere 
\end{equation*}
\end{proof}

\begin{corollary}
In Figure~\ref{fig:color} showing the triad of circles associated with $\triangle ABC$ and the inner
Apollonius circle externally tangent to each circle in the triad, we have that the sum of the yellow areas
is equal to the green area.
\end{corollary}

\begin{figure}[ht]
\centering
\includegraphics[scale=0.56]{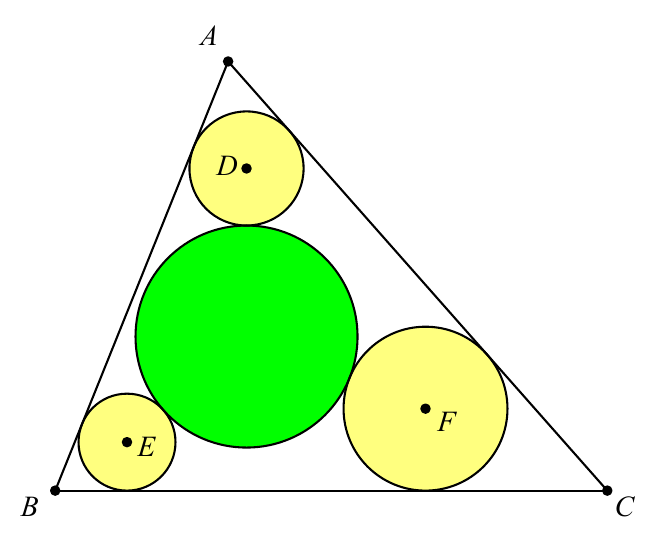}
\caption{yellow area = green area}
\label{fig:color}
\end{figure}

\begin{theorem}\label{thm5} 
Let $\rho_o$ be the radius of the outer Apollonius circle internally tangent to $\gamma_a$, $\gamma_b$, $\gamma_c$ (see Figure~\ref{fig:6}). We have
\begin{equation*}
   \rho_o=\frac{2}{3}\left(\rho_a+\rho_b+\rho_c\right)+\sqrt{\rho_a^2+\rho_b^2+\rho_c^2}.\\
\end{equation*}

\begin{figure}[ht]
\centering
\includegraphics[scale=1.1]{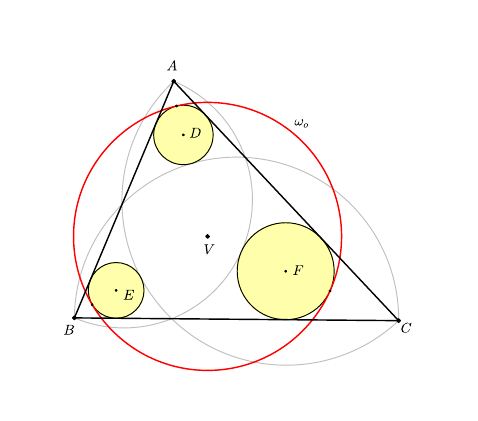}
\caption{}
\label{fig:6}
\end{figure}
\end{theorem}

\begin{proof}
The proof is similar to that of Theorem \ref{thm4}.
\end{proof}

Combining Theorem \ref{thm:Sangaku}, Theorem~\ref{thm4}, and Theorem~\ref{thm5}, we get the following nice results.

\begin{corollary}\label{cor2} The inradius $r$ and the radii $\rho_i$, $\rho_o$ of the Apollonius circles satisfy the relation
\begin{equation*}
   3\rho_o=\rho_i+4r.
\end{equation*}
\end{corollary}

\begin{corollary}\label{cor3} The radii $\rho_i$, $\rho_o$ of the Apollonius circles satisfy the relation
\begin{equation*}
   3\rho_o=2(\rho_a+\rho_b+\rho_c)+3\rho_i.
\end{equation*}
\end{corollary}

\textbf{Remark.} The centers $U$ and $V$ of the inner and outer Apollonius circles of $\gamma_a$, $\gamma_b$, and $\gamma_c$ are known ETC centers, 
namely $U=X(52805)$ and $V=X(52806)$.


\void{
\begin{theorem} \label{thm:7}
Let $\omega_i=(U,\rho_i)$ be the inner Apollonius circle, externally tangent to $\gamma_a$,
$\gamma_b$, and $\gamma_c$. Let $U_a$ be the touch point between $\omega_i$ 
and $\gamma_a$. Define $U_b$, $U_c$ cyclically (Figure~\ref{fig:17}). The triangles $ABC$ and $U_aU_bU_c$ are perspective. The perspector is $X_7$, the Gergonne point of $ABC$.

\begin{figure}[ht]
\centering
\includegraphics[scale=0.9]{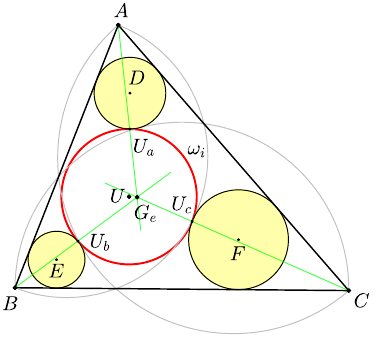}
\caption{}
\label{fig:17}
\end{figure}
\end{theorem}

\begin{proof}
The lines $AU_a$, $BU_b$, and $CU_c$ are concurrent by Theorem 2 of \cite{Rabinowitz}.
That theorem also shows that the point of concurrence is the  internal center of similitude
of the incircle of $\triangle ABC$ and the circle $\omega_i$.
The point of concurrence can be shown to be the Gergonne point, $G_e$, of $\triangle ABC$.
The proof is omitted.
\end{proof}
}

\begin{theorem}\label{thm:7}
Let $\omega_i$ be the inner Apollonius circle, externally tangent to $\gamma_a$,
$\gamma_b$, and $\gamma_c$. Let $U_a$ be the touch point between $\omega_i$
and $\gamma_a$. Define $U_b$ and $U_c$ cyclically.
Let $\omega_o$ be the outer Apollonius circle, internally tangent to $\gamma_a$,
$\gamma_b$, and $\gamma_c$. Let $V_a$ be the touch point between $\omega_o$
and $\gamma_a$. Define $V_b$ and $V_c$ cyclically. Let $G_e$ be the Gergonne point of $\triangle ABC$.
Then the points $A$, $V_a$, $U_a$, and $G_e$ are collinear.
Similarly, the points $B$, $V_b$, $U_b$, and $G_e$ are collinear;
and the points $C$, $V_c$, $U_c$, and $G_e$ are collinear (Figure~\ref{fig:17b}).
\end{theorem}

\begin{figure}[h]
\centering
\includegraphics[scale=0.75]{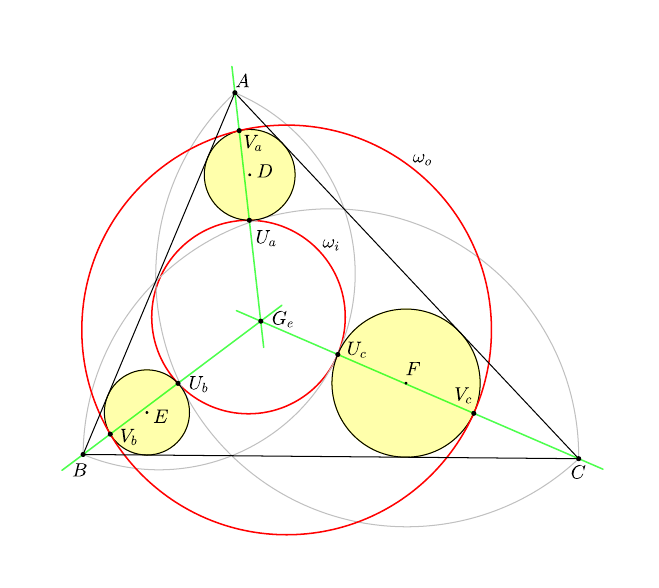}
\caption{}
\label{fig:17b}
\end{figure}

\begin{proof} Clearly, by symmetry, it is enough to prove that $A$, $V_a$, $U_a$, and $G_e$ are collinear.
From Theorem~\ref{thm2a}, we know that $G_e$ is the radical center of $\gamma_a$, $\gamma_b$,
and $\gamma_c$. Hence, from the Gergonne construction of Apollonius circles,
it follows that $U_a$, $V_a$, and $G_e$ are collinear.
Therefore, it remains to prove that $A$, $U_a$, and $G_e$ are collinear.
To this end we use barycentric coordinates.
We have $A=1:0:0$ and $G_e=\frac{1}{p-a}:\frac{1}{p-b}:\frac{1}{p-c}$.
The point $U_a$ divides the segment $DU$ joining the centers of the circles $\gamma_a$ and $\omega_i$ in the ratio $\rho_a:\rho_i$.
By using \textsc{Mathematica}, we find that
\begin{equation*}
   U_a=\frac{-a}{(p-b)(p-c)-\Delta}:\frac{1}{p-b}:\frac{1}{p-c}.
\end{equation*}
The points $A$, $U_a$, and $G_e$ are collinear because
\begin{equation*}
   \left| {\begin{array}{*{20}{c}}
1&0&0\\\\
\frac{-a}{(p-b)(p-c)-\Delta}&\frac{1}{p-b}&\frac{1}{p-c}\\\\
\frac{1}{p-a}&\frac{1}{p-b}&\frac{1}{p-c}
\end{array}} \right| = 0.
\end{equation*}
This completes the proof.
\end{proof}

\goodbreak
\textbf{Remark.}
We could also show that the lines $AU_a$, $BU_b$, and $CU_c$ are concurrent by using Theorem 2 of \cite{Rabinowitz}.
That theorem also shows that the point of concurrence, $G_e$, is the internal center of similitude
of the incircle of $\triangle ABC$ and the circle $\omega_i$.

\begin{corollary}
Let $U_a$ be the touch point between $\gamma_a$ and $\omega_i$
(Figure \ref{fig:equalTangents}).
Then the tangents from $U_a$ to $\gamma_b$ and $\gamma_c$ are equal.
\end{corollary}

\begin{figure}[ht]
\centering
\includegraphics[scale=0.5]{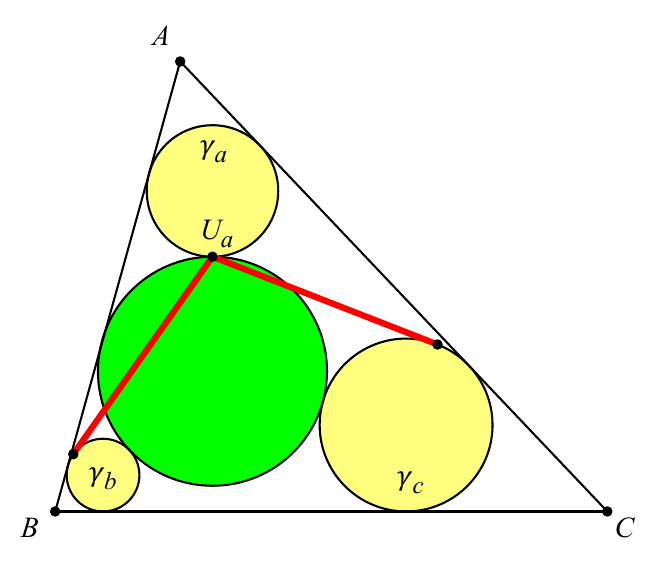}
\caption{red tangents are equal}
\label{fig:equalTangents}
\end{figure}

\begin{proof}
From Theorem~\ref{thm:7}, $AU_a$ is the Gergonne cevian from vertex $A$.
But from the proof of Theorem~\ref{thm2a}, this Gergonne cevian is the radical axis of
circles $\gamma_b$ and $\gamma_c$. Thus the two tangents have the same length.
\end{proof}

\void{
\begin{theorem}[to be deleted]
Let $\omega_i=(U,\rho_i)$ be the inner Apollonius circle, externally tangent to $\left(D\right)$, $\left(E\right)$, $\left(F\right)$. Let $U_a$ be the touch point between $\omega_i$ 
and $\gamma_a$. Define $U_b$, $U_c$ cyclically. The triangles $U_aU_bU_c$ and $DEF$ are orthologic.
The orthology center $(\triangle U_aU_bU_c,\triangle DEF)$ is the point $X_7$, the Gergonne point of $ABC$. 
\begin{figure}[ht]
\centering
\includegraphics[scale=1]{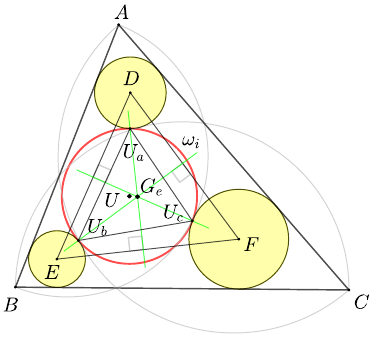}
\caption{}
\label{fig:8}
\end{figure}
\end{theorem}
}

\begin{theorem}
Let $\omega_i=(U,\rho_i)$, $\omega_o=(V,\rho_o)$ be the inner and outer Apollonius circles externally and internally tangent to $\gamma_a$,
$\gamma_b$, and $\gamma_c$, respectively. 
Let $U_a$, $V_a$ be the touch point of $\gamma_a$ with $\omega_i$ and $\omega_o$ respectively. Define $U_b$, $U_c$, $V_b$, $V_c$ cyclically. Let $I=X_1$, $G_e=X_7$ be the incenter and the Gergonne points of $\triangle ABC$ respectively (Figure~\ref{fig:19}). Then $U$ and $V$ lie on the Soddy line $IG_e$ and $UI:IV=3$.
\begin{figure}[ht]
\centering
\includegraphics[scale=1]{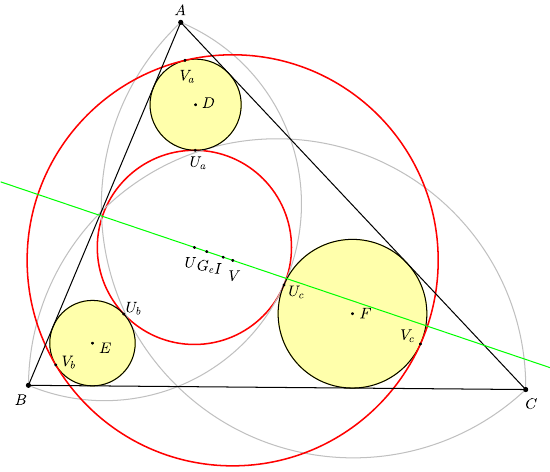}
\caption{}
\label{fig:19}
\end{figure}
\end{theorem}

\begin{proof}
This follows directly from the barycentric coordinates for $U$ and $V$.
\end{proof}

\section{Other properties}

\begin{theorem}
For the triad of circles associated with $\triangle ABC$,
let $x=\rho_a$, $y=\rho_b$, and $z=\rho_c$.
Let $u,v,w$ be the radii of the greatest circles inscribed in the circular segments shown in Figure~\ref{fig:10}.
Then
\begin{equation*}
   xw+yu+zv=xv+zu+yw.
\end{equation*}
\begin{figure}[ht]
\centering
\includegraphics[scale=0.8]{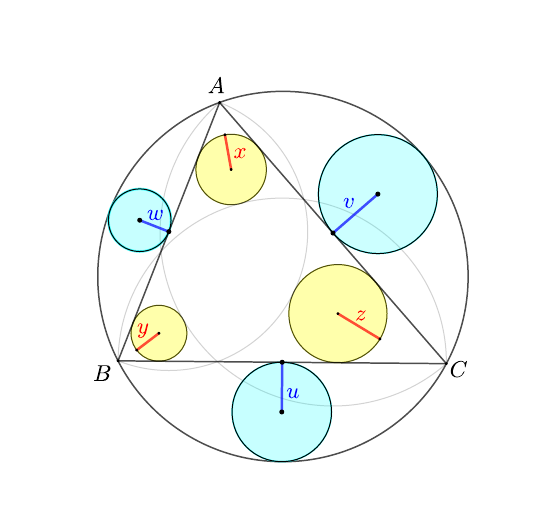}
\caption{}
\label{fig:10}
\end{figure}
\end{theorem}

\begin{proof}
From Theorem~\ref{thm:rho}, we have
\begin{equation*}
   x=r\left(1-\tan \frac{A}{2}\right),\quad y=r\left(1-\tan \frac{B}{2}\right),\quad z=r\left(1-\tan \frac{C}{2}\right).
\end{equation*}

\begin{figure}[ht]
\centering
\includegraphics[scale=0.8]{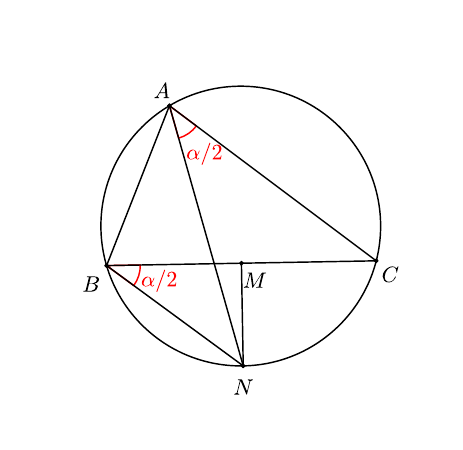}
\caption{}
\label{fig:21}
\end{figure}

On the other hand, we have (see Figure~\ref{fig:21})
\begin{equation*}
   u=\frac{1}{2}MN=\frac{1}{2}BM\cdot \tan \frac{A}{2}=\frac{a}{4}\tan \frac{A}{2}
\end{equation*}
and similarly
\begin{equation*}
   v=\frac{b}{4}\tan \frac{B}{2},\quad w=\frac{c}{4}\tan \frac{C}{2}.
\end{equation*}
Observe that 
\begin{equation*}
   \tan \frac{A}{2}=\frac{r}{s-a}=\frac{\Delta}{s(s-a)},\quad \tan \frac{B}{2}=\frac{\Delta}{s(s-b)}, \quad \tan \frac{C}{2}=\frac{\Delta}{s(s-c)}
\end{equation*}
so, using the Heron formula $\Delta^2=s(s-a)(s-b)(s-c)$, we get
\begin{align*}
   u(y-z)&=\frac{a}{4}\tan \frac{A}{2}\left(r\left(1-\tan \frac{B}{2}\right)-r\left(1-\tan \frac{C}{2}\right)\right)\\
         &=\frac{ar}{4}\tan \frac{A}{2}\left(\tan \frac{C}{2}-\tan \frac{B}{2}\right)\\
				 &=\frac{ar}{4}\left(\tan \frac{A}{2}\tan \frac{C}{2}-\tan \frac{A}{2}\tan \frac{B}{2}\right)\\
				 &=\frac{ar}{4}\left(\frac{\Delta}{s(s-a)}\frac{\Delta}{s(s-c)}-\frac{\Delta}{s(s-a)}\frac{\Delta}{s(s-b)}\right)\\
				 &=\frac{ar}{4}\left(\frac{s-b}{s}-\frac{s-c}{s}\right)\\
				 &=\frac{r}{4}\cdot\frac{ac-ab}{s}.
\end{align*}
Similarly, we have 
\begin{equation*}
  v(z-x)=\frac{r}{4}\cdot\frac{ba-bc}{s}, \qquad w(x-y)=\frac{r}{4}\cdot\frac{cb-ca}{s}. 
\end{equation*}
Therefore
\begin{align*}
    xw+yu+zv-(v+zu+yw)=&u(y-z)+v(z-x)+w(x-y)\\
		                  =&\frac{r}{4}\cdot\frac{ac-ab}{s}+\frac{r}{4}\cdot\frac{ba-bc}{s}+\frac{r}{4}\cdot\frac{cb-ca}{s}\\
											=&\frac{r}{4}\cdot\frac{ac-ab+ba-bc+cb-ca}{s}=0.
\end{align*}
This completes the proof.
\end{proof}


\goodbreak

\textbf{Errata.} On page 66, in equation (\ref{eq:15}), ``$t_2$'' should be ``$t_b$'' and in equation (\ref{eq:16}), ``$\cos\varphi_1$'' should be ``$\cos\varphi_a$''.


\begin{thebibliography}{99}

\newcommand{\blue}{\color{blue}}
\newcommand{\black}{\color{black}}


\bibitem{Andreescu}
Titu Andreescu and Oleg Mushkarov, Topics in Geometric Inequalities,
XYZ Press, 2019.

\bibitem{Capitan}
Francisco Javier Garcia Capit\'an, Problem 2428, Romantics of Geometry Facebook Group, October~2018.\\
\blue\url{https://www.facebook.com/groups/parmenides52/posts/1926266464153717/}\black

\bibitem{Honma}
Honma, ed., Zoku Kanji Samp\=o, Tohoku University Digital Collection, 1849.

\bibitem{ETC1123}
Clark Kimberling, \textit{Encyclopedia of Triangle Centers, entry for X(1123), the Paasche Point}.\\
\blue\url{https://faculty.evansville.edu/ck6/encyclopedia/ETCPart2.html#X1123}\black

\bibitem{ETC52805}
Clark Kimberling, \textit{Encyclopedia of Triangle Centers, preamble to entry X(52805),
Miyamoto-Moses Points}.\\
\blue\url{https://faculty.evansville.edu/ck6/encyclopedia/ETCPart27.html#X52805}\black

\bibitem{Okumura}
Hiroshi Okumura, Problems 2023--1, \textit{Sangaku Journal of Mathematics},
\textbf{7}(2023)9--12.\\
\blue\url{http://www.sangaku-journal.eu/2023/SJM_2023_9-12_problems_2023-1.pdf}\black

\bibitem{Paasche}
Ivan Paasche, Aufgabe P933: Ankreispaare, \textit{Praxis der Mathematik} \textbf{1}(1990)40.

\bibitem{Rabinowitz}
Stanley Rabinowitz, Pseudo-Incircles, \textit{Forum Geometricorum}, \textbf{6}(2006)107--115.\\
\blue\url{https://forumgeom.fau.edu/FG2006volume6/FG200612.pdf}\black

\bibitem{Ste} Milorad R.~Stevanovi\'{c}, Predrag B.~Petrovi\'{c}, and Marina M.~Stevanovi\'{c}, Radii of Circles in Apollonius's Problem,
\textit{Forum Geometricorum}, \textbf{17}(2017)359--372.\\
\blue\url{https://forumgeom.fau.edu/FG2017volume17/FG201735.pdf}\black

\bibitem{Sup}
Ercole Suppa and Marian Cucoanes, Solution of Problem 2023--1--6, \textit{Sangaku Journal of Mathematics},
\textbf{7}(2023)21--28.\\
\blue\url{http://www.sangaku-journal.eu/2023/SJM_2023_21-28_Suppa,Cucoanes.pdf}\black

\bibitem{Yiu}
Paul Yiu, Introduction to the Geometry of the Triangle, December 2012.\\
\blue\url{http://math.fau.edu/Yiu/YIUIntroductionToTriangleGeometry121226.pdf}\black

\end{thebibliography}
\end{document}